\documentclass[journal,twoside,web]{ieeecolor}
\usepackage{lcsys}
\usepackage{cite}
\usepackage{amsmath,amssymb,amsfonts}
\usepackage{algorithmic}
\usepackage{graphicx}
\usepackage{textcomp}
\def\BibTeX{{\rm B\kern-.05em{\sc i\kern-.025em b}\kern-.08em
    T\kern-.1667em\lower.7ex\hbox{E}\kern-.125emX}}
\newtheorem{thm}{Theorem}

\newtheorem{cor}{Corollary}
\newtheorem{Def}{Definition}

\begin{document}
\title{Sufficient Stability Conditions for  a Class of Switched Systems with Multiple Steady States}
\author{Jacopo Piccini$^{1}$, Elias August$^{1}$, Sigurdur Hafstein$^{2}$, and Stefania Andersen$^{2}$
\thanks{This work has received funding from European Union’s Horizon 2020 research and innovation programme under grant agreement no. 965417 and was supported in part by the Icelandic Research Fund under Grant 228725-051.}
\thanks{$^{1}$Jacopo Piccini and Elias August are with the Department of Engineering, Reykjavik University,
Menntavegur 1, 102 Reykjavik, Iceland (e-mail: jacopop@ru.is; eliasaugust@ru.is)}
\thanks{$^{2}$Sigurdur Hafstein and Stefania Andersen are with the Faculty of Physical Sciences, University of Iceland,
	Dunhagi 5, 107 Reykjavik, Iceland (e-mail: shafstein@hi.is; saa20@hi.is)}
}

\maketitle
\thispagestyle{empty}
\begin{abstract}
In this paper, we present a novel approach to determine the stability of switched linear and nonlinear systems using Sum of Squares optimisation. Particularly, we use Sum of Squares optimisation to search for a Lyapunov function that defines an absorbing set that confines solution trajectories. For linear systems, we show that this also implies global asymptotic stability. Using this approach, we can study stability for a broader range of switched systems, particularly, we can search for a global attractor for switched nonlinear systems, whose dynamics are given by polynomial vector fields and which have multiple equilibria or limit cycles.
\end{abstract}


\section{Introduction}
\label{sec:introduction}
\IEEEPARstart{S}{witched} systems are used for modelling in many different fields~\cite{Liberzon2003}. Examples for their use range from the biological sciences~\cite{ElFarra2005},
for example, the dynamics of human body thermoregulation during sleep abruptly changes
when sleep transitions from non-rapid eye movement (REM) sleep to REM sleep~\cite{Harding2020},
to mechanical engineering, where one example is the dynamics of an engine
with shifting gears~\cite{Shorten2007}. Switched systems can describe systems, whose dynamics are affected by instantaneous changes, by considering a set of continuous-time sub-systems and a rule governing the switching between them. Analysing switched systems is also important in the field of hybrid systems~\cite{Liberzon2003}.

We consider switched systems in continuous-time that are of the following form,
\begin{equation}
	\dot{x}= f_{\sigma}(x),\ \ x\in\mathbb{R}^n,\ \sigma\colon [0,\infty) \to \{1, 2, \ldots, N\}. \label{eq:Main}
\end{equation}
Vector $x$ denotes the state of the system and $\sigma$ is the switching signal;
given time-point $t$, the system dynamics is governed by function $f_i:\mathbb{R}^n\rightarrow \mathbb{R}^n$, where $\sigma(t)=i$.
That is, there are $N$ functions and which one is ``on" is defined by switching signal $\sigma$.
In this paper, we assume arbitrary switching, that is, switching signal $\sigma$ fulfils the technical assumption that there is only a finite number of discontinuity-points (switchings) on every finite time-interval, but is otherwise arbitrary.
Furthermore, we assume that the $f_i$ are linear or polynomial functions.

Determining stability of an equilibrium point for~\eqref{eq:Main}, that is of a point $x^*\in\mathbb{R}^n$, for  which $f_i(x^*)=0$ for $i=1,2,\ldots,N$, is often difficult, even for linear functions $f_i$~\cite{Liberzon1999}.
Stability results have been obtained when the dynamics of the switching are non-arbitrary, for instance, when a minimal
amount of time must pass before switching occurs~\cite{Hespanha1999, Shorten2007};
one then speaks of minimal dwell time. However, under arbitrary switching, the problem is notoriously hard and most results and methods used to determine stability of classical linear and nonlinear systems cannot be used. Thus, various works have adapted Lyapunov stability theory to switched systems and considered the construction of a Lyapunov functions for~\eqref{eq:Main}.

Under arbitrary switching, a necessary condition for stability of an equilibrium of~\eqref{eq:Main} is that it is stable for each subsystem given by $\dot{x}= f_i(x)$. Otherwise, if it is not stable for $\dot{x}= f_j(x)$ then it cannot be stable for~\eqref{eq:Main}; just set $\sigma(t)=j$ for all $ t$. Often, one seeks to find a common quadratic Lyapunov function~\cite{Shorten2007}, that is, to find a single quadratic Lyapunov function that guarantees the stability of each individual subsystem. If such a function exists then the equilibrium is also stable for the switched system.  In~\cite{Shorten1998, King2006}, algebraic conditions for the existence of a common quadratic Lyapunov function were defined. Arguably, those conditions on the system matrices seem rather restrictive.

A less conservative approach is to search for a common piecewise linear or polynomial Lyapunov function~\cite{Lin2009, Branicky1994}. In~\cite{Johansson1998}, a method for the computation of piecewise quadratic Lyapunov functions is presented. In the literature, there are mainly two approaches to compute such Lyapunov functions~\cite{Polanski1997, Hafstein2015}. One does so by means of either semidefinite programming or linear programming~\cite{Andersen2023}.

A different approach for computing a Lyapunov function is based on Sum of Squares (SOS) optimisation~\cite{Parrilo2003}. The advantage of using SOS optimisation is the ability, for systems whose dynamics are defined by polynomial vector fields, to search for Lyapunov functions that consist of polynomials of higher-order~\cite{Anderson2015}. In~\cite{Prajna2003}, the authors use SOS optimisation to search for polynomial Lyapunov functions and piecewise polynomial Lyapunov functions that guarantee global asymptotic stability for switched systems. In~\cite{CHESI2011621}, the author showed that, for linear switched systems, the existence of a homogenous polynomial Lyapunov function that is SOS is not only sufficient for global asymptotic stability of the equilibrium, but also that such a function,
of potentially high degree, must exist if the equilibrium is asymptotically stable.

In this paper, we present a novel method, based on SOS optimisation, to guarantee stability properties of nonlinear switched systems, whose dynamics are governed by polynomial vector fields. We do so by determining a globally absorbing set for the dynamical system given by~\eqref{eq:Main}. Using this approach, we can study  stability for a broader range of problems. For linear switched systems, we show that the existence of such a set translates to global asymptotic stability of the equilibrium at the origin.  As the Lyapunov function computed is not necessarily homogenous, one can potentially use polynomial Lyapunov function of lower degree than in~\cite{CHESI2011621}.

Other approaches have considered invariant sets as a tool to prove stability for switched systems. For example, in~\cite{Dorothy2016}, a stability theorem for switched systems is provided, where each subsystem possesses an invariant set. In~\cite{KUIAVA2013206} and in~\cite{Veer2020}, the notion of invariant sets and boundedness are invoked to prove stability of switched systems with multiple equilibria. However, all these results depend on nonzero dwell time.

The contributions of this paper are for switched systems, whose dynamics can be described by polynomial vector fields, and are the following. First, by relaxing the problem to determining a globally absorbing set for~\eqref{eq:Main} instead of proving global asymptotic stability directly by means of a common Lyapunov function, we can reduce the size of the problem and, thus, the computational effort by searching for Lyapunov functions that potentially consist of non-homogenous polynomials of lower order. Moreover, for switched systems with multiple equilibrium points or limit cycles, our approach provides novel means to establish stability and to characterise the attractor if it exists. This is an extension of the results for the same class of switched systems presented in~\cite{Prajna2003}, which provides stability certificates for switched system with a common equilibrium point for the subsystems. For linear systems, the approach presented in this paper scales better with system dimension than the one presented in~\cite{Andersen2023}, because no triangulation of the state-space is needed.

The remainder of the paper is organised as follows. Section~\ref{sec_abs_set} defines an absorbing set and shows how to search for one. Section~\ref{sec_stab} shows that for linear switched systems, the existence of an absorbing set translates to global asymptotic stability of the equilibrium point. In Section~\ref{examples}, we apply our method to different examples. Finally, we conclude the paper in Section~\ref{conc}.

\section{Obtaining an Absorbing Set}\label{sec_abs_set}
We investigate the stability of a switched system by searching for a certificate for ultimate boundedness of
the switched system and by using it to determine an absorbing set of~\eqref{eq:Main}. We do so by computing a Lyapunov function, $V(x)$, that is monotonically decreasing along all solution trajectories outside of the absorbing set. The notion of an absorbing set is  not only useful for dynamical systems with a unique equilibrium, but also for systems that have multiple equilibria, limit cycles, and/or strange attractors.

\subsection{Absorbing Sets}
An absorbing set is a special kind of positively invariant set. For this reason, we provide the following definitions~\cite{HKK00}.
\begin{Def}
	Let $x(t)$ be a solution of the dynamical system
	$\dot x=f(x)$, which we assume to exist for all $t$. Then $p$ is said to be a positive limit point of $x(t)$ if there is
	a sequence $\{t_n\}$, with $t_n\rightarrow\infty$ as $n\rightarrow\infty$, such that $x(t_n)\rightarrow p$ as
	$n\rightarrow\infty$. 
\end{Def}
\begin{Def}
	A set $M$ is invariant with respect to $\dot{x}=f(x)$ if
	\begin{equation*}
		x(0)\in M \Rightarrow x(t)\in M,\ \mathrm{for\ all}\ t\in\mathbb{R}.
	\end{equation*}
	This means, that if a solution belongs to $M$ at some time instant,
	then it belongs to $M$ for all time.
\end{Def}
\begin{Def}
	\label{def:newdef}
	For \eqref{eq:Main} and a switching signal $\sigma$ denote by $t\mapsto \phi_\sigma(t,x_0)$ its solution starting at $x(0) = x_0$. A set $M\subset \mathbb{R}^n$ is said to be \emph{positively invariant} for \eqref{eq:Main}, if for every switching signal $\sigma$ we have
	\begin{equation*}
		x_0\in M \Rightarrow \phi_\sigma(t,x_0)\in M,\ \mathrm{for\ all}\ t \ge 0.
	\end{equation*}
	A positively invariant set $M$ for~\eqref{eq:Main} is said to be an \emph{absorbing set} if additionally, for every compact set $C\subset \mathbb{R}^n$ there exist a time $t_C^*\ge 0$ such that for every switching signal $\sigma$
	\begin{equation*}
		x_0\in C \Rightarrow \phi_\sigma(t,x_0)\in M,\ \mathrm{for\ all}\ t \ge t_C^*.
	\end{equation*}
\end{Def}
\ \\
Now, for a switched system, the following theorem provides conditions
for set $\mathcal{B}$ to be a globally, uniformly  attractive, positively invariant set, that is, an absorbing set.

\begin{thm}\label{thm0}
	Assume that for \eqref{eq:Main} there exists a continuously differentiable function $V\colon \mathbb{R}^n\to \mathbb{R}$, a compact set $\mathcal{B}\subset \mathbb{R}^n$, and constants $\gamma\in \mathbb{R}$, $\rho>0$, such that
	\begin{align*}
		V(x)=\gamma\ \ \  & \text{for all $x\in\partial \mathcal{B}$,}\\
		V(x)> \gamma\ \ \  & \text{for all $x\in\mathbb{R}^n\setminus \mathcal{B}$,}\\
		\dot{V}(x) \le -\rho\ \ \  & \text{for all $x\in\mathbb{R}^n\setminus \mathcal{B}$.}
	\end{align*}
	Here
	\begin{align*}
		\dot{V}(x)&:=\max_{i=1,2,\ldots,N} \nabla V(x)\cdot f_i(x).
	\end{align*}
	Then $\mathcal{B}$ is an absorbing set for~\eqref{eq:Main}.
	
	\begin{proof}
		The proof of the theorem follows from standard arguments when using Lyapunov stability theory:
		As long as $x=\phi_\sigma(t,x_0) \notin  \mathcal{B}$ we have the condition
		that $V(x)$ monotonically decreases with time, $\dot{V}(x) \le -\rho$, which means that
		$$
		V(\phi_\sigma(t,x_0))-V(x_0) \le \int_0^t \dot{V}(\phi_\sigma(s,x_0)) ds \le -\rho t.
		$$
		First, it follows from this inequality that if $x_0\in \partial \mathcal{B}$, then the existence of a switching signal $\sigma$ and time $t>0$ such that $\phi_\sigma(s,x_0)\notin \mathcal{B}$ for $0<s \le t$ is impossible, because of $\gamma=V(x_0)<V(\phi_\sigma(t,x_0))$.
		Hence, $\mathcal{B}$ is positively invariant.
		
		Second, again by the inequality, $\mathcal{B}$ must be absorbing.  Indeed, fix a compact $C\subset \mathbb{R}^n$ and note that if $C\subset \mathcal{B}$ we can take $t^*_C=0$.  Otherwise denote by $\overline{V}$ the maximum value of $V$ on the closure of $C\setminus\mathcal{B}$ and set $t^*_C=(\overline{V}-\gamma)/\rho>0$ and note that for
		an arbitrary $x_0 \in C\setminus \mathcal{B}$ we have
		$$
		\rho t^*_C\ge V(x_0)-\gamma \ge  V(x_0)-V(\phi_\sigma(t,x_0)) \ge \rho t
		$$
		as long as $\phi_\sigma(t,x_0) \notin  \mathcal{B}$.  Thus, there exists
		a $t^*(x_0) \le t^*_C$ such that $ \phi_\sigma(t,x_0) \in \mathcal{B}$ for all $t\ge t^*(x_0)$ and $\mathcal{B}$ is absorbing.
	\end{proof}
\end{thm}

A dynamical system with an absorbing set is said to be ultimately bounded system~\cite{August2023}.
As shown previously,
if a dynamical system possesses an absorbing set, denoted by $\mathcal{B}$, then, not only will all trajectories starting within the absorbing set remain within the set, but also there exists for every $x_0\in \mathbb{R}^n$ a time point $t^*(x_0) \ge 0$ such that $\phi_\sigma(t,x_0) \in \mathcal{B}$ for all $t\ge t^*(x_0)$ and all switching signals $\sigma$.  Indeed, $t^*(x_0)$ can be chosen to imply the same for all solutions starting in a neighbourhood of $x_0$.

Throughout this paper, we use the notion of an absorbing set to obtain stability certificates for switched systems using SOS optimisation~\cite{Papachristodoulou2002, Anderson2015}. To solve SOS optimisation problems,  we use the MATLAB\textregistered~\cite{MATLAB} toolbox SOSTOOLS~\cite{Sostools}.
For more details on SOS optimisation and the SOS decomposition, see the Appendix.

\subsection{Obtaining an Absorbing Set}
To find a Lyapunov function that defines an absorbing set for switched system \eqref{eq:Main}, we apply the approach presented in~\cite{August2023} to all subsystems at once. The following two theorems provide means to obtain an absorbing set.
\begin{thm}\label{thm2}
	Given constants $\beta \geq 0$, and $\delta > 0$ and integer $\ell$, if there exists
	a radially unbounded SOS polynomial $V(x)$, $V(x)\geq\delta||x||_{2\ell}^{2\ell}$, where
	$||x||_{2\ell}^{2\ell}=x_1^{2\ell}+x_2^{2\ell}+\ldots+x_n^{2\ell}$,
	and a SOS polynomial $p_i(x)$ that solve the following SOS programme,
	\begin{equation}
		-f_i(x)\cdot \nabla V(x) -p_i(x)\left(\lVert x \rVert_2^2 - \beta\right) -\delta||x||_{2\ell}^{2\ell}\ \text{ is SOS}\ \forall i.
		\label{eq:AbsSet_ineq}
	\end{equation}
	then there exists $\tilde\beta >0$ such that with $\rho :=\delta \tilde\beta$ we have
	\begin{equation}\label{eq1_may}
		\dot{V}(x) \le -\rho,\ \ \text{for all $||x||_{2\ell}^{2\ell} \ge \tilde\beta$.}
	\end{equation}
	
	\begin{proof}
		It is shown in~\cite{August2023} that~\eqref{eq:AbsSet_ineq} implies that 
		\begin{equation}
			f_i(x)\cdot \nabla V(x)  \le -\delta \tilde\beta < 0,\ \text{for} \ ||x||_2^2 \geq \beta,
		\end{equation}
		and, thus,
		$||x||_{2\ell}^{2\ell} \geq \tilde\beta$, where $\tilde\beta$ is a positive constant, for all $i$.
		It follows that \eqref{eq1_may} holds.
	\end{proof}
\end{thm}

\begin{thm}\label{thm3}
	If there exist a positive constant $\gamma$ and a SOS polynomial $q(x)$ such that
	\begin{equation}\label{eq0_may13}
		-(V(x)-\gamma)+q(x)(||x||_2^2-\beta)\ \mathrm{is\ SOS}
	\end{equation}
	then $\mathcal{B} = \{x \in \mathbb{R}^n  | V(x)\leq\gamma \}$
	is an absorbing set of~\eqref{eq:Main}, where $V$ and $\beta$ are from a solution to \eqref{eq:AbsSet_ineq}.
	
	\begin{proof}
		If~\eqref{eq0_may13} holds, where $\beta$ solves~\eqref{eq:AbsSet_ineq}, then
		$||x||_2^2>\beta$ if $V(x) > \gamma$ and
		it follows from Theorem~\ref{thm0} and Theorem~\ref{thm2} that
		solution trajectories will approach the level set given by $V(x)=\gamma$ and, thus, that
		$\mathcal{B} = \{x \in \mathbb{R}^n  | V(x)\leq\gamma \}$
		is an absorbing set of~\eqref{eq:Main}.
	\end{proof}
\end{thm}

In this paper, we first increase the degree of Lyapunov function $V(x)$ (and accordingly of $p_i(x)$)
until we obtain a solution to~\eqref{eq:AbsSet_ineq}.
Then, we repeatedly solve~\eqref{eq:AbsSet_ineq} while decreasing $\beta\ge 0$ as much as possible,
which provides tighter bounds on the absorbing set.
Increasing the value of $\ell$ allows us to reduce the number of low-degree monomials
in $V(x)$; particularly, if the degree of $V(x)$ is $2\ell$ then it consists of a homogenous polynomial function.  
Note that solving~\eqref{eq:AbsSet_ineq} for $\beta = 0$ is equivalent to finding a common Lyapunov function
for the entire state space, as in~\cite{Prajna2003}.
If not stated otherwise, we set $\delta=1$.
Additionally, to reduce the size of $\mathcal{B}$, we solve the following problem instead of \eqref{eq0_may13},
\begin{align}\label{may11-1}
	&\text{minimise} \quad  \gamma\nonumber\\
	&\text{subject to}\quad -(V(x)-\gamma)+q(x)(||x||_2^2-\beta)\ \mathrm{is\ SOS}.
\end{align}
For an illustration of absorbing set $\mathcal{B}$, see Fig.~\ref{fig0}.

\section{Asymptotic Stability for Linear Systems}\label{sec_stab}
For the special case of linear switched systems 
we have the following theorem.
\begin{thm}\label{thm1}
	Linear switched systems with asymptotically stable sub-systems and a common equilibrium point, 
	given by $\dot x=A_ix$,
	are asymptotically stable under arbitrary switching if they possess an absorbing set.
\end{thm}
\begin{proof}
	Essentially, this follows from the homogeneity of linear systems.
	That is, for a given switching $\sigma$ and a constant $c>0$,
	we have for the solution to the switched linear system, that $\phi_\sigma(t,cx_0)=c\phi_\sigma(t,x_0)$.
	In the following, $\mathcal{B}_r$ denotes the open ball centred at the origin with radius $r>0$.
	
	First, we show stability of the origin. Let constant  $\varepsilon>0$ be given.  Choose the constant $c>0$ so large that  $\mathcal{B}\subset\mathcal{B}_{c\varepsilon}$
	and choose the constant $\delta >0$ so small that $\mathcal{B}_{c\delta}\subset\mathcal{B}$.
	Since $\mathcal{B}$ is positively invariant for the switched system, we have for every  switching that the solution $t\mapsto \phi_\sigma(t,x_0)$ stays in $\mathcal{B}_{c\varepsilon}$ for all $t\ge0$, whenever $x_0\in\mathcal{B}_{c\delta}\subset \mathcal{B} $.
	Thus, $\phi_\sigma([0,\infty),\mathcal{B}_{c\delta}) \subset \mathcal{B}_{c\varepsilon}$ and by the homogeneity property $$
	\phi_\sigma([0,\infty),\mathcal{B}_{\delta}) \subset \mathcal{B}_{\varepsilon}.
	$$
	That is, stability follows by the homogeneity property. For an illustration of above argumentation, see Fig.~\ref{fig0}.

	To conclude the proof, we show that $\lim_{t\to \infty}\phi_\sigma(t,x_0)=0$ for every switching and every $x_0\in\mathbb{R}^n$.  Assume, by way of contradiction, that this does not hold true. Then, there exists a switching signal $\sigma$, an $x_0\in\mathbb{R}^n$, an $\varepsilon>0$, and an unbounded sequence of times $(t_n)$, such that $\|\phi_\sigma(t_n,x_0)\|_2> \varepsilon$ for all $n\in \mathbb{N}$.  Choose the constant $c>0$ so large that  that $\mathcal{B}\subset\mathcal{B}_{c\varepsilon}$.  By the definition of $t^*_C$ with $C=\overline{\mathcal{B}_{c\varepsilon}}$, cf.~Definition \ref{def:newdef}, we have $\phi_\sigma(t,cx_0)\in \mathcal{B}\subset \mathcal{B}_{c\varepsilon}$ for all $t\ge t^*_C$.  By the homogeneity property it follows that $\phi_\sigma(t,x_0) \in \mathcal{B}_{\varepsilon}$ for all $t\ge t^*_C$, which is a contradiction and we have proved the theorem.
	
\end{proof}

Finally, Theorem~\ref{thm1} has an obvious corollary.
\begin{cor}\label{cor1}
	Solutions of linear switched switched systems with asymptotically stable sub-systems given by
	$\dot x=A_ix$
	cannot have a periodic solution for any switching $\sigma$ if the switched system as a whole has an absorbing set. 
\end{cor}
\begin{figure}[hbtp]
	\hspace{-24mm}
	\vspace{-8mm}
	\includegraphics[scale = 0.3]{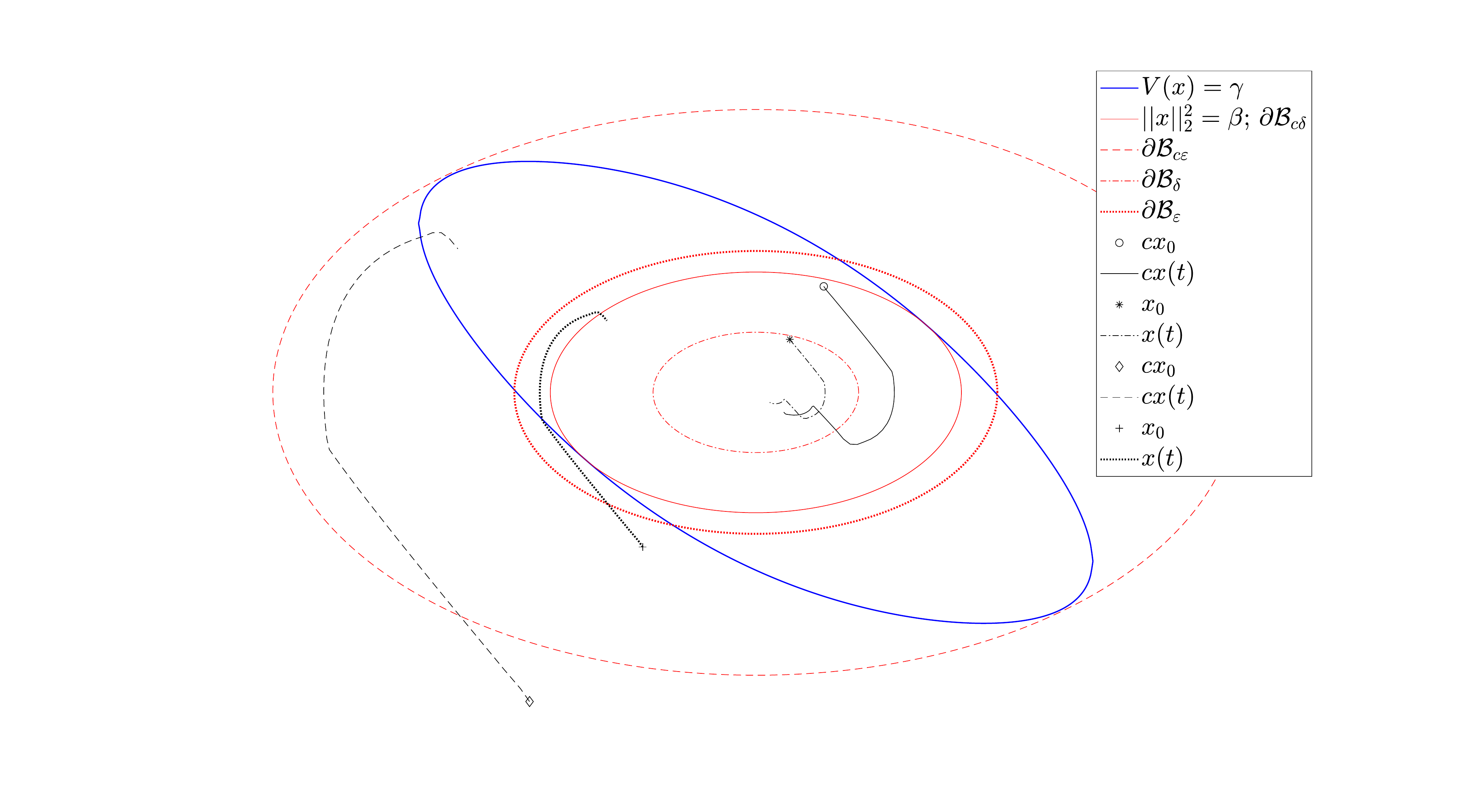}
	\caption{The figure shows $\partial\mathcal{B}$, which is given by $V(x)=\gamma$, and
		its relation to $||x||_2^2=\beta$ (see main text). Moreover, it depicts two exemplary trajectories used
		to prove the first part of Theorem~\ref{thm1} (`$\circ$' \& `$*$') and two used in the second part,
		where we choose $c$ sufficiently large such that, for $t>t_C^*$, the trajectory starting 
		at `$\diamond$' is and remains in $\mathcal{B}_{c\varepsilon}$, which implies that
		the trajectory starting at `$+$' is and remains in $\mathcal{B}_\varepsilon$ for $t>t_C^*$.}
	\label{fig0}
\end{figure}

\section{Examples}\label{examples}
In this section, we present a few examples. All problems are solved on a MacBook Pro with a 2.3 GHz quad-core Intel Core i5 processor. Furthermore, we solve linear matrix inequalities using semidefinite programming, for which we use YALMIP~\cite{Yalmip}, a MATLAB\textregistered~\cite{MATLAB} toolbox. We solve SOS optimisation problems using SOSTOOLS~\cite{Sostools}, another MATLAB\textregistered\ toolbox. In both cases, we solve the problems using the solver SeDuMi~\cite{Sturm99}.

\subsection{Linear Switched System}\label{ex1_lin}
The first problem is interesting, since it has been often investigated
in the literature~\cite{Andersen2023} and shows the potential of solving SOS optimisation problems.
The problem consists of determining the asymptotic stability of a planar linear switched system given by
\begin{equation}\label{eq_lin_b}
	\dot{x}= A_{\sigma(t)}x,\ \sigma(t)\in\{ 1,2\},
\end{equation}
where
\begin{equation}
	A_1= \left[\begin{array}{cc}0&1\\-0.1&-2\end{array}\right],\
	A_2= \left[\begin{array}{cc}0&1\\-b&-2\end{array}\right]. \label{eq:2d_ab}
\end{equation}
First, note that by solving the following linear matrix inequality,
\begin{equation}
	P = P^\mathrm{T} \succ 0,\ PA_i + A_i^\mathrm{T}P \prec 0,\ i = 1,2,
\end{equation}
one obtains a common quadratic Lyapunov function up to $b \leq 5.36$.

Now, we let $\beta=0$ and the degree of $V(x)$ be $2\ell$, that is, we search for a common homogenous polynomial Lyapunov function that  guarantees global asymptotic stability. We solve~\eqref{eq:AbsSet_ineq} for $\beta=0$ (and $\delta=0.001$) for increasingly larger values of $\ell$, which allows us to increase the value of $b$, however, not beyond $b=12$, which is reached for $\ell=6$.
Since the existence of a homogenous polynomial that is SOS, of sufficiently high degree ($\ell>>1$), and solves~\eqref{eq:AbsSet_ineq} is a necessary condition for stability~\cite{CHESI2011621}, we continue increasing the value of $\ell$. However, solving~\eqref{eq:AbsSet_ineq} for $\ell\geq10$ leads to numerical problems. The Lyapunov function, that we obtain, is given by
\begin{align*}
	V(x) &=   1326.8x_1^{12} + 3997.0355x_1^{11}x_2 + 13366x_1^{10}x_2^2 \\
	&+22545x_1^9x_2^3 + 24318x_1^8x_2^4 + 17999x_1^7x_2^5 \\
	&+ 10097x_1^6x_2^6 + 4333.6x_1^5x_2^7 + 1379.2x_1^4x_2^8 \\
	&+ 304.99x_1^3x_2^9 + 44.607x_1^2x_2^{10} + 3.9466x_1x_2^{11} \\
	&+ 0.1836x_2^{12}.
\end{align*}
It follows from Theorem~\ref{thm1} that the switched system presented in this example is globally asymptotically stable for $b=12$.
Finally, $b=12$ much improves the previously reported result of $b=5.36$ obtained by solving a linear matrix inequality, particularly, as simulations show that the switched system becomes unstable for $b=13.26$.

\subsection{Affine Switched System with 2 Equilibrium Points}\label{ex1_lin_mul}
In this example, we consider the switched dynamical system given by
\begin{equation}\label{eq_lin_mul}
	\dot{x}= A_{\sigma(t)}x+d_{\sigma(t)},\ d_1=0,\
	d_2= \left[\begin{array}{c}1\\1\end{array}\right],\ \sigma(t)\in\{ 1,2\},
\end{equation}
where $A_1$ and $A_2$ are given by~\eqref{eq:2d_ab} for $b=2$.
Note that the equilibrium point of one subsystem is the origin, while for the other one it is not.
For $\ell=2$ and $\beta=3.3$, we obtain a solution for~\eqref{eq:AbsSet_ineq}.

Specifically, we obtain a homogenous polynomial Lyapunov function of degree 4 that is given by
\begin{align}\label{V_lin_mul}
	V(x) &= 436.8x_1^4 + 929.2x_1^3x_2 + 963.1x_1^2x_2^2 \nonumber\\
	&+ 519.2x_1x_2^3+168.1x_2^4.
\end{align}
Using this value for $\beta$, we solve~\eqref{may11-1} to obtain the boundary of absorbing set $\mathcal{B}$, given by $V(x)=\gamma$, where $\gamma=8725$ (see Fig.~\ref{fig1}).
\vspace{-2mm}
\begin{figure}[h]
	\hspace{0mm}
	\vspace{-4mm}
	\includegraphics[scale=.215]{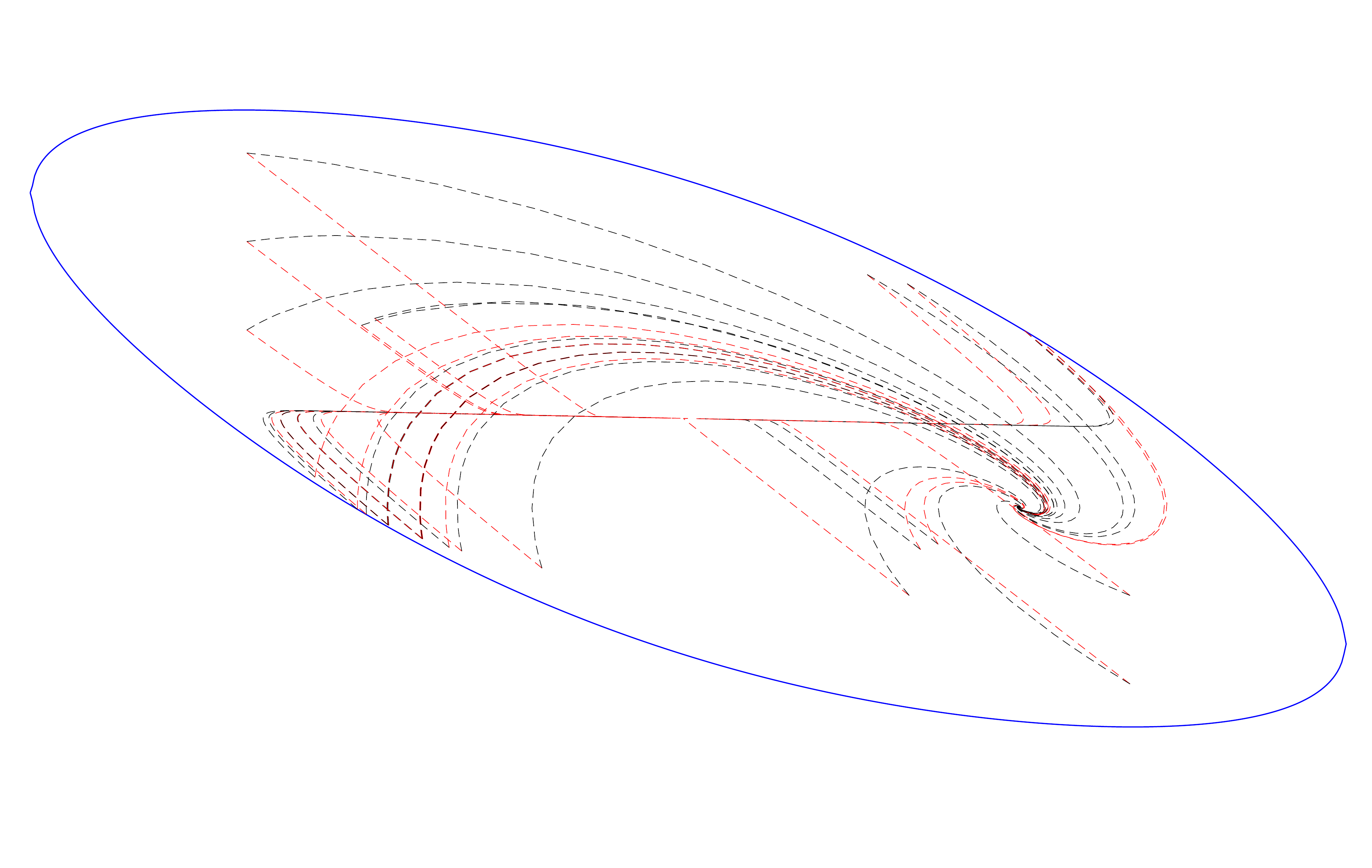}
	\caption{The figure shows the boundary of the absorbing set of~\eqref{eq_lin_mul}
		with a few exemplary system trajectories of $\dot x=A_1x$ and  $\dot x=A_2x+d_2$. The boundary is given by
		$V(x)=8725$ and $V(x)$ is given by~\eqref{V_lin_mul}.The figure lies in the area $[-3, 3] \times [-4, 4]$.}
	\label{fig1}
\end{figure}
\subsection{Affine Switched System with 3 Equilibrium Points}
Here, we compare our approach to the one in~\cite{Alpcan2010}, which depends on dwell time, when applied to Example 4.2 in~\cite{Alpcan2010}, which considers the following switched system that has multiple equilibria,
\begin{equation}\label{eq:alpacan}
	\dot{x}= Ax + b_{\sigma(t)},\  \sigma(t) \in \{1, 2, 3\},
\end{equation}
where
\begin{equation*}
	A = \left[\begin{array}{cc}	-1 & -1 \\ 1 & -1	\end{array}\right],
	b_1 = \left[\begin{array}{c}	         1 \\ 1 \end{array}\right]
	b_2 = \left[\begin{array}{c}	        -1 \\ 1 \end{array}\right]
	b_3 = \left[\begin{array}{c}		 1 \\ -1 \end{array}\right].
\end{equation*}
Solving~\eqref{eq:AbsSet_ineq}, with $\ell=2$, and $\beta=2$, we find a fourth order homogenous common Lyapunov function given by
\begin{align*}
	V(x) &= 8.7957x_1^4 + 1.8977x_1^3x_2 + 17.4811x_1^2x_2^2\\
	&- 1.5706x_1x_2^3  + 9.3477x_2^4.
\end{align*}
We then solve~\eqref{may11-1}, for $\beta=2$ and obtain $\gamma=38.43$, which guarantees boundedness
of system solutions independent of dwell time (see Fig.~\ref{fig5}); in~\cite{Alpcan2010} the average dwell time was required to be bounded away from zero.  
\vspace{-3mm}
\begin{figure}[h]
	\hspace{-7mm}
	\vspace{-7mm}
	\includegraphics[scale=.235]{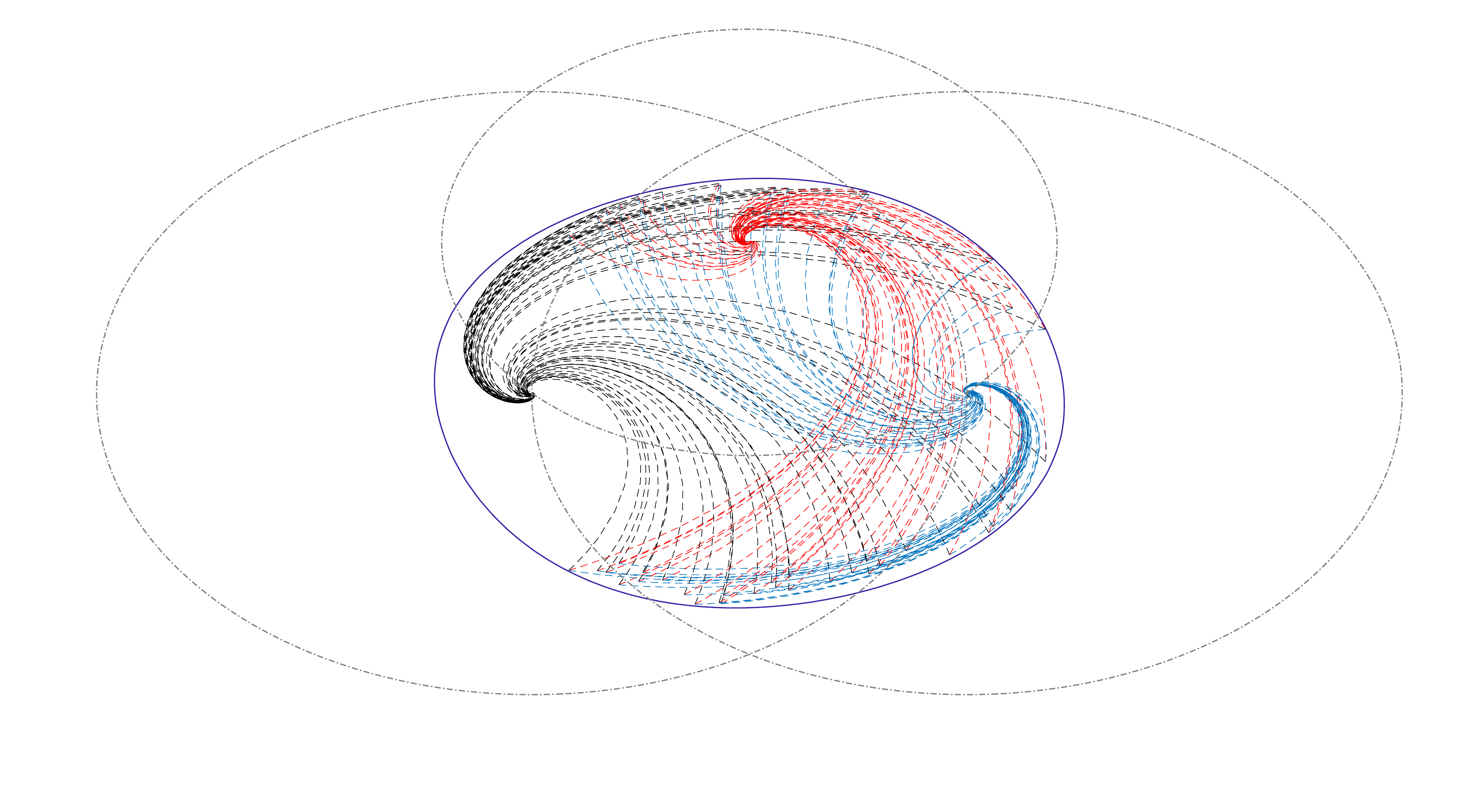}
	\caption{The figure shows exemplary trajectories of~\eqref{eq:alpacan} and the boundary of its absorbing set. We overlap the bounded absorbing set obtained using our approach with the absorbing sets, depicted in grey dash-dotted lines, from~\cite{Alpcan2010}. The figure lies in the area $[-3.5, 3.5] \times [-3.5, 3.5]$.}
	\label{fig5}
\end{figure}

\subsection{Nonlinear Switched System with Unique Equilibrium}\label{ex2_nonlin}
Consider the nonlinear switched system given by
\begin{equation}\label{eq_ex2_nonlin}
	\dot{x}= A_{\sigma(t)}(x)x,\ \sigma(t)\in\{ 1,2\},
\end{equation}
where
\begin{align*}
	A_1(x)&= \left[\begin{array}{ccc} 0.2868-x_2^2&    1.5387&    0.1731\\
		-0.3628&    0.0893&   -0.6175\\
		0.0892&    1.2898&   -1.4316\end{array}\right],\\
	A_2 &= \left[\begin{array}{ccc} -1.5007  &  1.3875&   -0.4402\\
		0.4919&   -1.5442&    0.1360\\
		0.2914&   -0.4561&    0.0231\end{array}\right].
\end{align*}
First, to solve~\eqref{eq:AbsSet_ineq} for $\beta=0$ and, thus, to guarantee asymptotic stability, $V(x)$ must be a SOS polynomial of degree 6. Note that we set $\ell=2$. On the other hand, solving \eqref{eq:AbsSet_ineq} and, then, \eqref{may11-1}, for $\beta=5$, we can
guarantee boundedness of solutions with a Lyapunov function $V(x)$ that is a SOS polynomial of degree 4 (see Fig.~\ref{fig3D}). Significantly, by doing so, the size of the resulting semidefinite programme reported by SOSTOOLS goes down from having $1339$ equalities and $262$ decision variables to having $444$ equalities and $125$ decision variables for solving~\eqref{eq:AbsSet_ineq} and
$116$ equalities and $41$ decision variables for solving~\eqref{may11-1}.
\begin{figure}[h]
	\hspace{-1mm}
	\vspace{-4mm}
	\includegraphics[scale=.235]{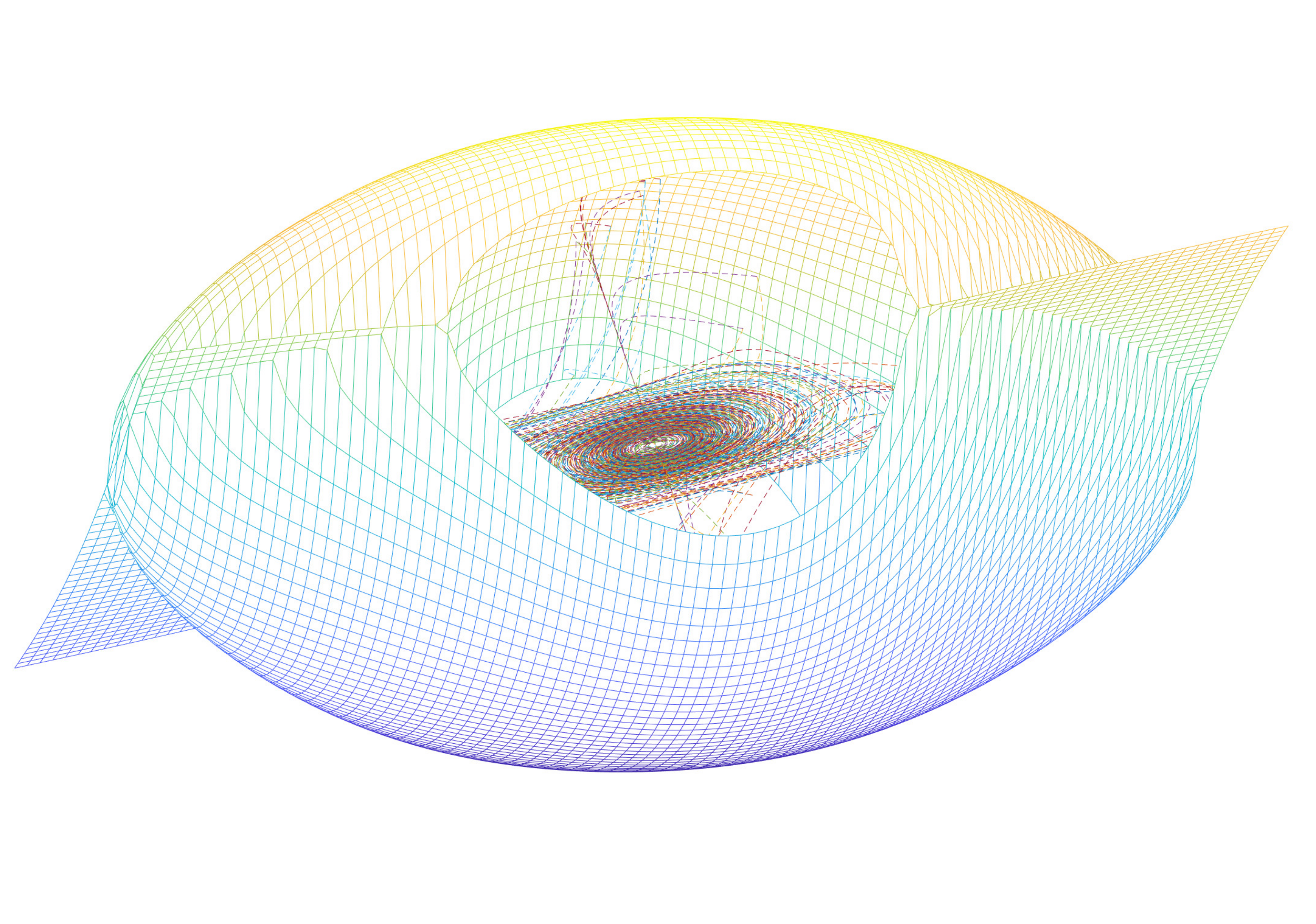}
	\caption{The figure shows the boundary of the absorbing set of \eqref{eq_ex2_nonlin}
		with a few exemplary system trajectories of $\dot x=A_1(x)x$ and  $\dot x=A_2x$. The figure lies in the space $[-5, 5] \times [-2, 2] \times [-3, 3]$.}
	\label{fig3D}
\end{figure}

\subsection{Nonlinear Switched System with Limit Cycle}\label{nonlin_vdp}
In this example, we apply our approach to the analysis of a nonlinear switched system, where one subsystem consist of a Van der Pol oscillator. The system is given by
\begin{equation}\label{eq_vdp}
	\dot{x}= f_{\sigma(t)}(x),\ \sigma(t)\in\{ 1,2\},
\end{equation}
where
\begin{equation}
	\begin{aligned}
		f_1(x) &= \left[\begin{array}{c}
			x_2 \\ -x_1 -(x_1^2 - 1)x_2\end{array}\right],\\
		f_2(x) &= \left[\begin{array}{c}
			x_2 \\-6x_1 - 2x_2 \end{array}\right].	
	\end{aligned}\label{eq_nonlin1}
\end{equation}

Note that subsystem $f_1(x)$ admits a limit cycle around the origin, which is an unstable equilibrium point of the subsystem.
However, we can show that the system possess an absorbing set with the origin in its interior. Specifically, for $\ell=1$, $\delta = 0.0001$, and $\beta = 14$, solving~\eqref{eq:AbsSet_ineq} and, then,~\eqref{may11-1}, we can bound system trajectories.
Specifically, we obtain a Lyapunov function of degree 6 that defines the boundary of the absorbing set $\mathcal{B}$ (see Fig.~\ref{fig2}).
\vspace{-2mm}
\begin{figure}[h]
	\hspace{-3mm}
	\vspace{-3mm}
	\includegraphics[scale=.225]{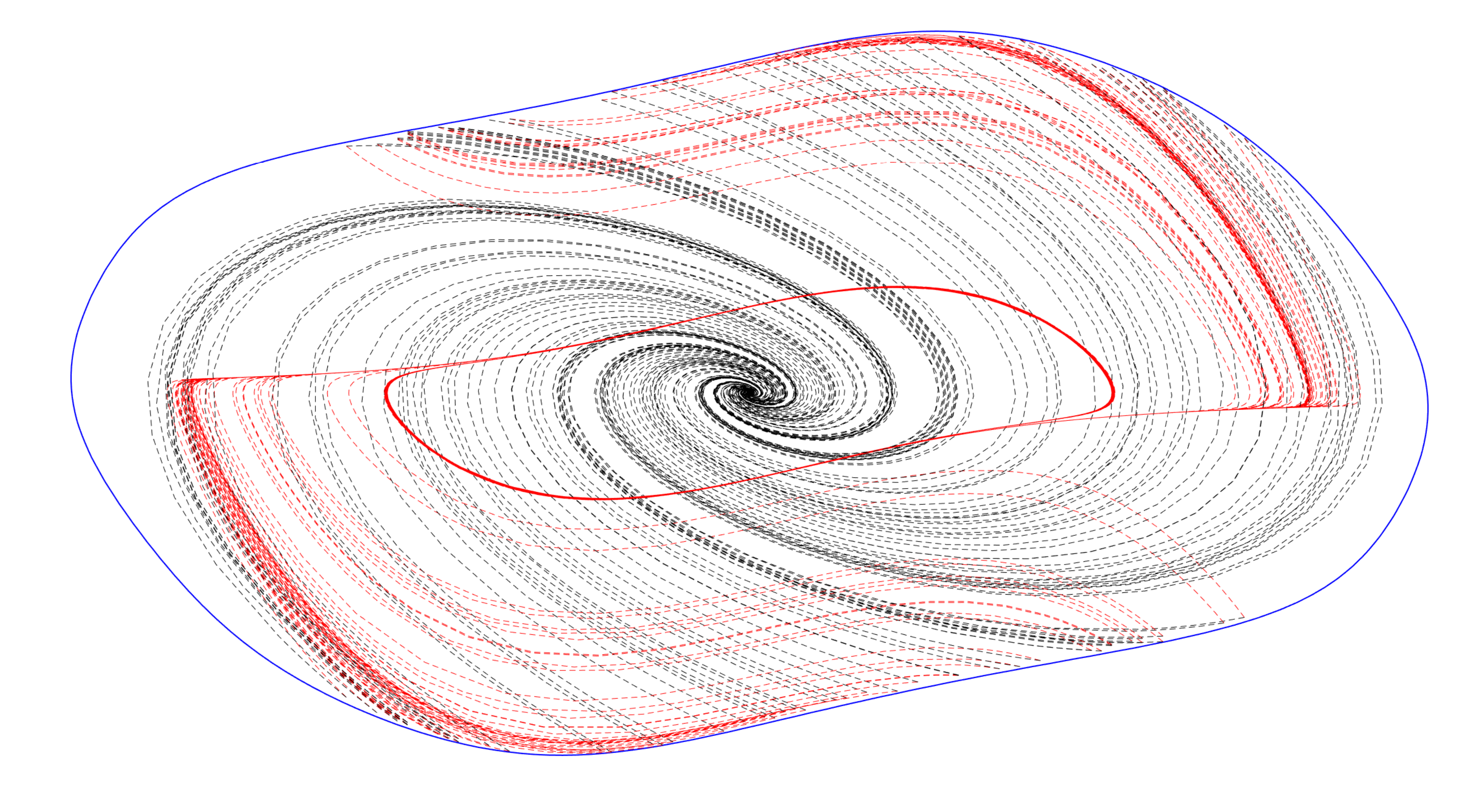}
	\caption{The figure shows the boundary of the absorbing set of~\eqref{eq_vdp}
		with a few exemplary system trajectories. The figure lies in the area $[-3.8, 3.8] \times [-9.5, 9.5]$.}
	\label{fig2}
\end{figure}

\section{Conclusions}\label{conc}
Using SOS optimisation, in this paper, we presented a novel approach to provide stability certificates for switched linear and nonlinear systems, whose dynamics are described through polynomial vector fields, under arbitrary switching. We did so, by providing means to search for certificates of ultimate boundedness of the switched system under consideration. We also showed that, for linear switched systems, the existence of an absorbing set implies global asymptotic stability.
We applied the presented approach to different examples to illustrate that it allows to locate an absorbing set, even when the switched system is composed of subsystems with distinct equilibrium points or possessing limit cycles. Furthermore, we showed that if guaranteeing boundedness of solutions is sufficient, as opposed to guaranteeing global asymptotic stability of an equilibrium, then our novel method might achieve this with reduced computational effort.

\appendix
\subsection*{Sum of Squares Decomposition}
Consider the real-valued polynomial function $F(x)$ of degree $2d$,
$x\in\mathbb{R}^n$. 
A sufficient condition for $F(x)$ to be nonnegative is that it can be decomposed into a SOS~\cite{Parrilo2003}:
$F(x)=\sum_i f_i^2(x)\geq0$,
where $f_i$ are polynomial functions.
$F(x)$ is a SOS if and only if there exists a positive semidefinite 
matrix $R$ and $F(x)=\chi^\mathrm{T}R\chi$, where
$$\chi^\mathrm{T}=\left[\begin{array}{cccccccc} 1 &x_{(1)} &\ldots& x_{(n)} &x_{(1)}x_{(2)}&\ldots& x^d_{(n)} \end{array}\right].$$
The entries of vector $\chi$ consist of all monomial combinations of the elements of vector $x$ up to degree $d$ (including $x_{(i)}^0=1$) and, thus, its
length is $\ell=\binom{n+d}{d}$. Note that
$R$ is not necessarily unique. However, $F(x)=\chi^\mathrm{T}R\chi$ 
poses certain constraints
on $R$ of the form tr$(A_jR)=c_j$, where $A_j$ and $c_j$ are
appropriate matrices and constants respectively. 
In general, in order to find $R$, we solve the optimisation problem
associated with the following semidefinite programme:
\begin{eqnarray*}\label{sos0}
	\mathrm{min}\ &&\mathrm{tr}(A_0R) \nonumber\\
	\mathrm{s.t.}\ &&\mathrm{tr}(A_jR)=c_j,\ j=1,\ldots,m \nonumber\\
	&& R=R^\mathrm{T}\succeq0.
\end{eqnarray*}

\section*{Acknowledgment}
The authors would like to thank the editor and the anonymous reviewers for their valuable comments, which substantially improved the paper.

\bibliography{cdc2023_bib}
\bibliographystyle{IEEEtran}

\end{document}